\documentclass[10pt,a4paper]{amsart}

\usepackage{amsfonts}
\usepackage{amsthm}
\usepackage{amssymb}
\usepackage{latexsym}
\usepackage{amsmath}
\usepackage{setspace}
\usepackage{amscd}
\usepackage{verbatim}
\usepackage{graphicx}
\usepackage{graphics}
\usepackage[matrix,arrow]{xy}
\usepackage[active]{srcltx}
\usepackage{mathrsfs}
\usepackage{overpic}
\usepackage{tikz}

\usepackage
[hypertexnames=false,backref=page,pdfpagemode=UseNone,breaklinks=true,extension=pdf,colorlinks=true,linkcolor=blue,
urlcolor=blue]{hyperref}

\newtheoremstyle{sanspoint}
  {\topsep}   
  {\topsep}   
  {\itshape}  
  {}          
  {\bfseries} 
  {}          
  {5pt plus 1pt minus 1pt} 
  {}          

\theoremstyle{sanspoint}

\newtheorem*{thm}{{\sc Theorem}}

\title[]{On the Finiteness of Anosov flows on $3$-manifolds via Contact Geometry}
\dedicatory{To Yasha Eliashberg, on the occasion of his 80th 
birthday, whose
ideas and energy have paved our way.}

\author{Jonathan Bowden}
\address{J. Bowden, Leibniz University Hannover, Welfengarten 1, D-30167 Hannover, Germany}
\email{jonathan.bowden@math.uni-hannover.de}

\author{Vincent Colin}
\address{V. Colin, Nantes Universit\'e, CNRS, Laboratoire de Math\'ematiques Jean Leray, LMJL, F-44000 Nantes, France}
\email{vincent.colin@univ-nantes.fr}

\date{\today}

\makeatletter
\@namedef{subjclassname@1991}{Mathematics Subject Classification (2020)}
\makeatother

\newcommand{\R}{\mathbb{R}}
\newcommand{\Z}{\mathbb{Z}}
\newcommand{\N}{\mathbb{N}}

\theoremstyle{plain}
\newtheorem{theorem}{\sc Theorem}[section]

\newtheorem{lemma}[theorem]{\sc Lemma}
\newtheorem{corollary}[theorem]{\sc Corollary}
\newtheorem{conjecture}[theorem]{\sc Conjecture}

\theoremstyle{definition}
\newtheorem{definition}[theorem]{\sc Definition}

\theoremstyle{remark}
\newtheorem{remark}[theorem]{\sc Remark}

\theoremstyle{question}
\newtheorem{question}[theorem]{\sc Question}

\theoremstyle{problem}
\newtheorem{problem}[theorem]{\sc Problem}

\theoremstyle{example}
\newtheorem{example}[theorem]{\sc Example}

\begin{document}

\maketitle

\begin{abstract} 
In dimension $3$, following Eliashberg-Thurston and Mitsumatsu, one can associate a transverse pair of oppositely oriented contact structures to any Anosov flow. We show that 
when the ambient manifold is hyperbolic, the isotopy classes of these contact structures completely determine the flow up to isotopy orbit equivalence. This then implies that  on any closed hyperbolic $3$-manifold, the number of Anosov flows modulo isotopy orbit equivalence 
is finite. This approach also yields an explicit bound on the number of orbit equivalence classes of Anosov flows in terms of the number of universally tight contact structures. In addition, we show finiteness of pseudo-Anosov flows for which the complement of the singular orbits is atoroidal and, in the non-hyperbolic case, we obtain a control on the dynamics of Anosov flows on the hyperbolic pieces of the JSJ-decomposition.

\end{abstract}

\vspace{.2cm}



\section{Introduction}

In this article, we prove the following:

\begin{theorem}[Finiteness Conjecture: Hyperbolic Case] \label{thm: finiteness} A closed hyperbolic $3$-manifold carries only finitely many Anosov flows up to isotopy orbit equivalence. 
\end{theorem}
Our argument provides a bound for the number of isotopy orbit equivalence classes, given by the sum, over homotopy classes of plane fields $\mathcal{H}$, of products of the numbers of isotopy classes of universally tight positive and negative contact structures in $\mathcal{H}$. 
The number of universally tight contact structures up to isotopy on any closed hyperbolic $3$-manifold is known to be finite since the work of Colin-Giroux-Honda \cite{CGH},  and it can in principle be explicitly bounded by combinatorial data extracted from a triangulation of $M$, following \cite{CGH}. There is however no universal bound on the number of Anosov flows, since by a recent work of Bowden and Mann \cite{BoM}, see also B\'eguin-Yu \cite{BeY}, for every $n\in \N$, there exists a closed  hyperbolic $3$-manifold that has at least $n$ non orbit equivalent Anosov flows.

Our result takes place in a long standing effort to understand Anosov flows, which are the paradigm of hyperbolic dynamics, starting with works of Anosov, Fried, Ghys, Barbot, Fenley, Mann, Barthelm\'e, Dehornoy...
In particular, a similar finiteness result was known on Seifert fibered spaces \cite{Gh, Ts}.
The approach to this problem has recently been renewed by considering its interactions with contact geometry, as initiated by the work of Mitsumatsu \cite{Mit}, Eliashberg-Thurston \cite{ETh} and Marty \cite{Ma}. Exploiting these new techniques, finiteness results have been obtained by Barthlem\'e-Bowden-Mann for {\it Reeb} Anosov flows \cite{BBM} or more recently by Barthelm\'e-Tsang-Zung for pseudo-Anosov flows without {\it perfect fits} on closed $3$-manifolds \cite{BaTZ} or with {\it skew pieces} by Barthelm\'e-Paulet \cite{BaP}.

These contributions involve highly sophisticated contact homology techniques and the comparison of Anosov and Reeb dynamics. In the present work, we at the same time get rid of all hypothesis on the Anosov flows and simplify the proofs by only using classical $3$-dimensional contact topology, and in particular neither pseudo-holomorphic curves nor Reeb dynamics.
Our techniques moreover complete the work of Bowden-Massoni \cite{BoM} and conduct to a classification of isotopy orbit equivalence classes of Anosov flows in contact geometry terms.

\begin{theorem}\label{theorem: classification} On a closed hyperbolic $3$-manifold $M$, the following are equivalent:
\begin{enumerate} 
    \item[(a)] two Anosov flows $X$ and $X'$ are isotopy orbit equivalent;
    \item[(b)] the associated contact structures $\xi_+$, $\xi_+'$ and $\xi_-$, $\xi_-'$ are pairwise isotopic;
    \item[(c)]  the contact pairs $(\xi_+,\xi_-)$ and $(\xi'_+,\xi'_-)$ are deformation equivalent among transverse contact pairs;
    \item[(d)] $X$ and $X'$ are homotopic through projectively Anosov flows.
\end{enumerate}
\end{theorem}

This result uses our Theorem \ref{thm: homotopy}, where we prove that the isotopy classes of $\xi_+$ and $\xi_-$ determine the free homotopy data of periodic orbits of $X$, together with the results of Bowden-Massoni \cite{BoM}. Both results hold for any (possibly toroidal) closed $3$-manifold $M$. To deduce orbit equivalence, we then apply the orbit rigidity criterion of Barthelm\'e-Frankel-Mann \cite{BFM}, which only holds under certain additional assumptions, that are automatic if $M$ is hyperbolic.
In particular, in Theorem \ref{theorem: classification},
  we could replace the hypothesis ``$M$ is hyperbolic" by the more general situation when $M$ is toroidal, $X$ is transitive and has no transverse incompressible tori or Klein bottles, see \cite[Theorem 6.1.9]{BaM}.



In the general toroidal case, our techniques yield restrictions on the dynamics of Anosov flows on hyperbolic pieces in the following sense:
\begin{theorem} \label{thm: finiteness_JSJ} Let $M_{hyp} \subseteq M$ be a hyperbolic piece of the JSJ-decomposition. The subsets of conjugacy classes of $\pi_1(M_{hyp}) \subseteq \pi_1(M)$ that are realised by periodic orbits of an Anosov flow is finite. 
\end{theorem}

In the pseudo-Anosov case we obtain finiteness for flows such that the complement of the singular orbits is atoroidal, which includes the case that the flow has no perfect fits, see \cite{BaTZ}:
\begin{theorem}\label{thm:finiteness_pA_case}
    On any closed hyperbolic $3$-manifold, the number of pseudo-Anosov flows such that the complement of the singular orbits is atoroidal 
    is finite up to isotopy orbit equivalence.
\end{theorem}

We see Theorems \ref{thm: finiteness_JSJ} and \ref{thm:finiteness_pA_case} as building blocks toward a proof of the Finiteness Conjecture for transitive pseudo-Anosov flows on arbitrary closed $3$-manifolds.
Cutting manifolds into JSJ-pieces, one still needs to understand the effect of  fibered Dehn twists/Lutz modifications along incompressible tori, as given in Theorem \ref{theorem: toroidal}/\ref{theorem: generation} from \cite{CGH}. This will be the subject of future investigation, together with the other research directions highlighted in Section \ref{section: questions}.


\subsection{Outline of Proof}

An Anosov flow $X$ on a $3$-manifold $M$ comes equipped with a pair $(\xi_+,\xi_-)$ of positive and negative universally tight contact structures that intersect transversally along $\langle X\rangle$. We show in Theorem \ref{thm: homotopy} that, given a free homotopy class of loops $\gamma$ in $M$, the pulled-back contact structures $(\xi_+^\gamma,\xi_-^\gamma)$ on the infinite  cover $M^\gamma :=\widetilde{M}/\langle \gamma\rangle \simeq S^1\times \R^2$ of $M$ detect whether $\gamma$ is, up to orientation, represented by a periodic orbit of $X$ or not.
The discriminating property is the pair of {\it asymptotic slopes} for both $\xi_+^\gamma$ and $\xi_-^\gamma$ of this open solid torus: they are {\it constantly equal at infinity} if and only if $\gamma^{\pm 1}$ is represented by a closed orbit.
By  Theorem \ref{theorem: BFM} of Barthelm\'e-Frankel-Mann \cite{BFM},
this data is enough to determine the class of $X$ modulo isotopy orbit equivalence when $M$ is hyperbolic. 
Thus, on a closed hyperbolic $3$-manifold, the number of isotopy orbit equivalence classes of Anosov flows is bounded by the number of classes of pairs of positive and negative universally tight contact structures $(\xi_+,\xi_-)$, each being considered independently modulo isotopy.
To conclude, we appeal to a classification theorem of Colin-Giroux-Honda \cite{CGH}: on a closed hyperbolic $3$-manifold, there are finitely many tight contact structures up to isotopy. This approach also applies, more or less directly, to the pseudo-Anosov case in the situation that the complement of the singular orbits is atoroidal as well via a doubling trick.

In the toroidal case, the finiteness of tight contact structures with vanishing Giroux torsion, which is always a property of $\xi_\pm$, is only up to fibered Dehn twists along incompressible tori and isotopies, and this is one of the main reasons why we do not directly get a proof of the finiteness conjecture in this situation. 



\subsection*{Acknowledgments}

We thank Thomas Barthelm\'{e}, Pierre Dehornoy, Katie Mann, Thomas Massoni, Chi Cheuk Tsang and Jonathan Zung for numerous and precious discussions, indications, and comments.

This material is based on work supported by the National Science Foundation under Grant No. DMS-1928930, while the authors were in residence at the Simons Laufer Mathematical Sciences Institute in Berkeley, California, during the spring semester of 2026.
JB is supported by the Heisenberg-Program (BO 4423/4-1) of the German Science Foundation. VC is supported by the Institut Universitaire de France.

\subsection*{Relation to other work} While this text was being finalised, Gabai and Li announced a proof of the finiteness of Anosov and pseudo-Anosov flows on closed hyperbolic $3$-manifolds using  different methods.


\section{Contact geometry in Dimension 3}
Let $M$ be an oriented $3$-manifold. A {\it positive (resp. negative) contact structure} on $M$ is a plane field $\xi$ locally defined as the kernel of a non vanishing differential $1$-form $\alpha$ such that $\alpha\wedge d\alpha >0$ (resp. $\alpha\wedge d\alpha <0$).

A key point in contact topology is to understand the interactions of contact planes with curves and surfaces contained in $M$.
A curve in a contact manifold $(M,\xi)$ is {\it Legendrian} if it is everywhere tangent to $\xi$.
If a Legendrian curve $\gamma$ is null homologous in $(M,\xi)$ and if $\xi$ is coorientable, its {\it Thurston-Bennequin invariant} is the linking number between $\gamma$ and $\gamma_\epsilon$, where $\gamma_\epsilon$ is a small push-off of $\gamma$ in some transverse direction to $\xi$. If there exists a Legendrian unknot whose Thurston-Bennequin invariant is $0$, then the contact structure $\xi$ is said to be {\it overtwisted}. Otherwise it is {\it tight}. A contact manifold $(M,\xi)$ is universally tight whenever the pull-back of $\xi$ on the universal cover of $M$ is tight (and virtually overtwisted otherwise). As a keystone of the theory, the standard contact manifold $(\R^3,\ker (dz+xdy))$ is tight by a theorem of Bennequin \cite{Be}.

If $\xi$ is an oriented plane field on an oriented $3$-manifold $M$, the {\it characteristic foliation} $\xi (S)$ of an oriented surface $S\subset (M,\xi)$ is the singular foliation generated by the singular (oriented) line field $\xi \cap TS$. 
When $\xi$ is a contact structure, a surface $S\subset (M,\xi)$ is {\it convex} if it is transverse to a {\it contact vector field} $X$, that is a vector field whose flow preserves $\xi$. This fundamental notion introduced by Giroux has important features, see \cite{Gi1}.
First, convexity is a $C^\infty$ generic property of surfaces. Second, a convex surface $S$ is equipped with a {\it dividing set} $\Gamma_S$, which is a non-empty multicurve in $S$ defined as the set of points $x \in S$ where $X(x)\in \xi(x)$. This multicurve is a crucial data that essentially determines the germ of $\xi$ near $S$ up to isotopy. It is transverse to the characteristic foliation of $S$ and divides $S$ into regions where the characteristic foliation is alternatively expanding (i.e. dilating some area form) and contracting (i.e. contracting an area form). 

Moreover, any characteristic foliation drawn on $S$ and that is similarly {\it divided} by $\Gamma_S$ can be realized as the characteristic foliation $\xi(S')$ of a surface $S'$ isotopic to $S$ by a $C^0$ small isotopy transverse to $X$ and relative to $\Gamma_S$, a property called {\it Giroux's flexibility Lemma}.
When the surface $S$ is a torus and when the ambient contact manifold is tight, the dividing set is made of an even number of essential (parallel) curves. Giroux's flexibility Lemma then tells that the torus $S$ can be isotoped relatively to $\Gamma_S$ to a $C^0$ close one $S'$ whose characteristic foliation is non singular, each component of $S'\setminus \Gamma_{S'}=S'\setminus \Gamma_S$ containing exactly one periodic orbit of $\xi(S')$, alternatively attracting and repelling.

\subsection{Giroux Normal Form} When we consider a thickened torus $(T^2\times [0,1],\xi)$ where $\xi$ is a universally tight contact structure that makes the boundary convex, then a theorem of Giroux \cite[Proposition1.8]{Gi2} shows that $\xi$ can be put in normal form by isotopy with respect to the foliation $(T^2\times \{t\})_{t\in [0,1]}$. Each torus $T^2\times \{t\}$ is then either convex, or transverse to $\xi$ with a characteristic foliation 
that is a suspension. 

These tori get organised into subintervals of $[0,1]$. Whenever we have an interval of convex tori, the slope of the dividing curves remain constant, and whenever we have an interval of suspension characteristic foliations, the slope is monotonically increasing or decreasing with $t$, depending on the sign of $\xi$.  In the case that the slopes are constant we call the contact structure {\em non-rotative}.

Moreover, if the thickened torus $(T^2\times [0,1],\xi)$ is embedded in a tight contact manifold $(M,\xi)$ where it becomes compressible
(e.g. it is a layer of concentric tori in a solid torus), then the slopes of the dividing sets of the characteristic foliations can never be meridional (as otherwise we would get an overtwisted structure).
We will exploit these properties in the rest of the paper.

 \subsection{Coarse classification of tight contact structures}

Given a contact manifold $(M,\xi)$ and an embedded torus $T\subset M$ that is transverse to $\xi$, an equation of $\xi$ in a small well-chosen tubular neighbourhood $T_{(x,y)}\times [-1,1]_t$ of $T=T\times \{0\}$ is of the form $$\xi =\text{Ker} (\cos (f(x,y,t))dx-\sin(f(x,y,t))dy),$$ where the (positive) contact condition is given by $\frac{\partial f}{\partial t}>0$.
In this situation, a {\it k-Lutz modification of $\xi$} is defined by replacing $\xi$ on $T\times [-1,1]$ by $$\text{Ker} (\cos(f(x,y,t)+k\pi (t+1))dx-\sin(f(x,y,t)+k\pi (t+1))dy),$$ together with a smoothing at the boundary. That is we add $k$ full twists to $\xi$ around the $t$-axis.

Such a modification can have very different effects depending on the characteristic foliation of the torus $T$.
If it is a suspension, then the characteristic foliations on the tori $T\times \{t\}$ are all suspensions and their slopes rotate $k$ full extra turns.
If the characteristic foliation of $T$
is made of Reeb components pointing in the same direction, then the characteristic foliations keep a constant slope and move as $t$ varies by translation transversally to the slope direction in one direction or another, depending on the way the Reeb components bend.

To keep track of the effect of the first type of behavior, Giroux introduced the notion of {\it torsion}.

Consider the contact structure on $T^2 \times \R=(\R/\Z)_{(x,y)}^2\times \R_t$ defined by $$\xi_k = \text{Ker}(\cos(kt)dx-\sin(kt)dy).$$
A contact structure $(M,\xi)$ is said to have at least {\em Giroux $k$-Torsion} if there is a contact embedding of 
 $$( T^2 \times [0,2\pi], \xi_k) \hookrightarrow (M,\xi).$$

It is shown in \cite{CGH} that Lutz modifications are the only possible source of infinite families of tight contact structures.

\begin{theorem}\label{theorem: generation}\cite[Th\'eor\`eme 7]{CGH}
On a closed contact $3$-manifold, there exists a finite family of tight contact structures $\xi_1,\dots, \xi_n$ and for each contact structure $\xi_k$, $1\leq k\leq n$, a finite collection of tori $T^k_1,\dots,T^k_{i_k}$ transverse to $\xi_k$, such that every tight contact structure on $M$ is, up to isotopy, obtained by Lutz modifications on one of the $\xi_k$ along the tori $T^k_1,\dots,T^k_{i_k}$.
\end{theorem}
Note that the tori $T^k_1,\dots,T^k_{i_k}$ are usually not disjoint, but all together form a branched surface so that we can still perform Lutz modifications at will despite the intersections.

The next step of \cite{CGH} was to exploit the fact that, though not correct as roughly stated now, performing a Lutz modification on a torus whose characteristic foliation is made of Reeb components essentially produces a contact structure that is conjugated to the initial one by a fibered Dehn twist on $T\times [-1,1]$. Notice here that when the torus is compressible and the manifold irreducible, such a Dehn twist is isotopic to the identity.
On the other hand, if a Lutz modification is operated on a torus whose characteristic foliation is a suspension, then it increases the Giroux torsion. In this case, if the torus is compressible, it immediately introduces an overtwisted disk (a torus $T\times \{t\}$ will have a characterisitic foliation with meridional slope). In particular, an atoroidal tight contact $3$-manifold never has Giroux torsion.

Summarizing, the final conclusions of \cite{CGH} that we will use in this work are the following results.

\begin{theorem}\label{theorem: toroidal}\cite[Th\'eor\`eme 6]{CGH}
Given $N\in \N$, a closed $3$-manifold carries finitely many tight contact structures of Giroux torsion less than $N$ up to isotopy and fibered Dehn twists along incompressible tori.
\end{theorem}

In particular, we have the important special case:

\begin{theorem}\label{theorem: atoroidal}\cite[Th\'eor\`eme 2]{CGH}
An atoroidal closed $3$-manifold carries finitely many tight contact structures up to isotopy.
\end{theorem}


\section{Anosov Flows and bi-contact structures}
\textbf{Assumption:} We shall assume that all bundles are oriented. This will not affect the generality of our approach, since our arguments characterize the homotopy data of periodic orbits on certain intermediate coverings 
where all bundles can be assumed to be orientable: cf.\ Remark \ref{rem:Finite_Index}. 

\subsection{Anosov Flows} A nonsingular flow\footnote{By slight abuse we will identify the flow with $X$, which is not a problem as all vector fields are assumed to be smooth.}  $(X_t)$ generated by a smooth vector field $X$ is \textit{Anosov} if the tangent bundle of $M$ has a (continuous) flow-invariant splitting
 $$TM =   E^{ss} \oplus \langle X \rangle \oplus E^{uu}.$$
such that there are constants $C, a >0$ for which the inequalities
$$
    \Vert d\varphi_t(v^{s})\Vert \leq C e^{-a t} \Vert v^{s} \Vert, \qquad
    \Vert d\varphi_t(v^{u})\Vert \geq C^{-1} e^{a t} \Vert v^{u}\Vert
$$
hold for all $ t \ge0$ and all $v^{u}\in E^{uu}, v^{s} \in E^{ss}$. Here we have chosen some auxiliary Riemannian metric.

The subbundles $E^{uu}, E^{ss}$ are called the {\em strong unstable} and {\em strong stable} directions of the flow, respectively. 
In this case the $2$-plane fields 
$$ E^{ws}= E^{ss} \oplus \langle X \rangle, \qquad  E^{wu} = E^{uu} \oplus \langle X \rangle,$$
are tangent to foliations $\mathcal{F}^{ws}, \mathcal{F}^{us}$ respectively; these are called the {\em weak stable} and {\em weak unstable} 
foliations of the flow, respectively.

\subsection{Bifoliated Planes}
    Given an Anosov flow $X$ on a closed $3$-manifold $M$ one can consider the lift to the universal cover. Barbot \cite{Bar} and independently Fenley \cite{Fen} proved that the quotient $\widetilde{M}/\widetilde{X} = \mathcal{P} \cong \R^2$ is a plane and the weak (un)stable foliations descend to the quotient to give a transverse pair of foliations $\mathcal{F}_s,\mathcal{F}_u$. The action of the fundamental group on the universal cover then descends to the flow space and preserves the (1-dimensional) leaves of the bifoliated plane. For further details on the structure of these foliations see \cite{BaM}.

    It is a classical fact of Hirsch-Pugh-Shub \cite{HPS} that the weak (un)stable bundles are of class (at least) $C^1$ and hence the projected foliations are tangent to transverse $C^1$ vector fields.
    \begin{example}[Skew-plane] Consider the subset of $\R^2$ given by
    $$\mathcal{P_{\R}} = \{(x,y) \ | \ -1 < x-y < 1\}$$
    with bifoliation given by horizontal and vertical segments. This is called a (positively) skewed bifoliated plane. An Anosov flow whose associated flow space is (oriented) homeomorphic to a skewed bifoliated plane is called skew.
    \end{example}
   \begin{figure}[h]
\centering
\scalebox{1.1}{
\begin{tikzpicture}[x=1cm,y=1cm]

\draw[->,line width=0.4pt] (1,0) -- (5,0);
\draw[->,line width=0.4pt] (0,1) -- (0,5);

\draw[dashed,line width=0.4pt] (1.5,0.5) -- (4.5,3.5);
\draw[dashed,line width=0.4pt] (0.5,1.5) -- (3.5,4.5);

\draw[blue,line width=0.4pt] (0.5,1.5) -- (2.5,1.5);

\draw[blue,line width=0.4pt] (1,2) -- (3,2);

\draw[red,line width=0.4pt] (1.5,0.5) -- (1.5,2.5);

\draw[blue,line width=0.4pt] (1.5,2.5) -- (3.5,2.5);

\draw[blue,line width=0.4pt] (2,3) -- (4,3);

\draw[red,line width=0.4pt] (2.5,1.5) -- (2.5,3.5);

\draw[blue,line width=0.4pt] (2.5,3.5) -- (4.5,3.5);

\draw[red,line width=0.4pt] (3.5,2.5) -- (3.5,4.5);

\node[below] at (5.2,0) {$\mathcal{L}^u$};
\node[left] at (0,5.2) {$\mathcal{L}^s$};

\end{tikzpicture}
}
\caption{The skew-plane $\mathcal{P}_\R$ with (horizontal) stable and (vertical) unstable leaves. (credit: \cite{BaM})}
\label{fig:skew_plane}
\end{figure}
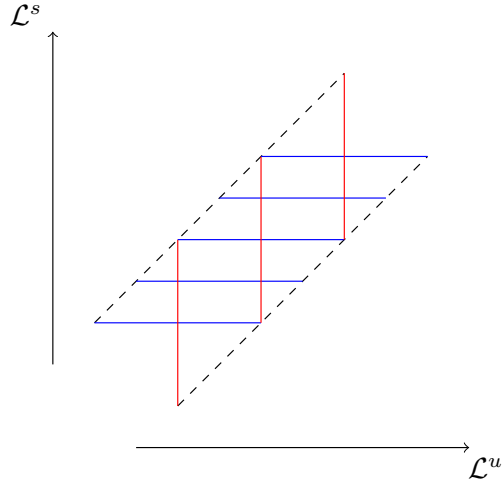
It is a fundamental fact that an Anosov flow is skew or a suspension if and only if its weak foliations are $\R$-covered in the sense that on $\tilde{M}$ the leaf space is Hausdorff and hence homeomorphic to $\R$ (cf. \cite{Fen,Bar}). Since we only consider hyperbolic manifolds below, we will not need to consider the suspension case.

In the case that the bifoliated plane is neither homeomorphic to the skew plane nor the product (affine) plane, then we say that the flow is {\it non-$\R$-covered}.

\subsection*{Lozenges and Scalloped Regions} There are two important types of regions of a bifoliated plane that one can consider.

 A {\it lozenge} denotes a $4$-gon embedded inside the bifoliated plane, whose sides are alternatively tangent to the stable and unstable foliations and which meet at two distinct ideal vertices, see \cite[Definition 2.4.5]{BaM}. Lozenges are called {\it adjacent} if they share a side. 

A {\it scalloped region} is a $4$-gon with punctures on the boundary, including at its $4$ vertices, that is embedded in the bi-foliated plane and whose $4$ sides consist of infinite sequences of either stable or unstable leaves meeting at ideal points  
such that its interior is homeomorphic to a bifoliated square. For a more precise definition, we refer to \cite[Definition 2.5.7]{BaM}.


\subsection{Bi-contact Structures and Projectively Anosov Flows} 
Eliashberg-Thurston \cite{ETh} and Mitsumatsu \cite{Mit} considered a weakening of the Anosov condition and defined a flow to be \textit{projectively Anosov} if there is a flow-invariant splitting 
$$TM =   E^{ss} \oplus \langle X \rangle \oplus E^{uu}.$$
and constants $a,C >0$ as above such that for all $t \ge 0$ and $v^{u}\in E^{uu}, v^{s} \in E^{ss}$ the following holds:
$$
    \frac{\Vert d\varphi_t(v^{s})\Vert}{ \Vert d\varphi_t(v^{u})\Vert} \leq C e^{-a t} \frac{\Vert v^{s} \Vert}{\Vert v^{u}\Vert}. 
$$
In this case one also has (flow-invariant) weak stable and unstable distributions, which will not be integrable in general. Moreover, a fundamental fact is that such a projectively Anosov flow is equivalent to a pair of transverse oppositely oriented contact structures that are tangent to the flow.

Mitsumatsu \cite{Mit} and Eliashberg-Thurston \cite{ETh} observed that one can associate a transverse pair of contact structures $(\xi_+,\xi_-)$ tangent to any projectively Anosov flow, the one positive, the one negative. This is explicitly given by
$$ \xi_+ = \Delta^{+} \oplus \langle X \rangle \ , \ \xi_-= \Delta^{-} \oplus \langle X \rangle,$$
where $\Delta^{+}, \Delta^{-} \subseteq E^{ss} \oplus E^{uu}$ denote the diagonal and an anti-diagonal with respect to a metric for which all line fields are orthogonal\footnote{To be precise this will only be of class $C^0$, so one needs to smoothen these line fields.}. Any such pair is called a {\it bi-contact structure} and up to deformation through such pairs it is well defined for a fixed flow. Finally a projectively Anosov flow is equivalent to the existence of a bi-contact structure  that is tangent to the flow (cf.\ \cite[Proposition 2.2.3]{ETh}). 

\subsection{Explicit $1$-Forms}
Let $\eta_s,\eta_u$ be 1-forms of class $C^1$ that are dual to $\mathcal{F}_u,\mathcal{F}_s$. Then by making appropriate choices (cf.\ \cite[Lemma 3.5]{Mas}) the bi-contact structure associated to an Anosov flow is given by
$$\xi_\pm = \text{Ker}(\eta_s \mp \eta_u).$$
Then note that pushing forward under the Legendrian flow given by $X$ we obtain
$$\lim_{t \to \infty}(X_t)_*\xi_{\pm} = \pm\mathcal{F}_u \ , \ \lim_{t \to \infty}(X_t)_*\xi_+ = \mathcal{F}_s.$$
Where the (co)orientations coincide for the negative limit and are opposite in forward time. In particular, we obtain a `Front-Back' completion of $\widetilde{M} \cong \mathcal{P} \times (-\infty,\infty)$ by adding front and back at infinity
$$\mathcal{P} \times [-\infty,\infty]$$
and the bi-contact structure extends (continuously).
\subsection{Bifoliated Planes and Cones}\label{subsection:bifol_cone}
In terms of the bifoliated plane, if we project the contact structures, then we see that the positive contact structure maps to the negative quadrants, where $Q_+ := \{{\color{blue}x}{\color{red}y} > 0\}$ and the negative contact structure maps to the positive quadrants $Q_- := \{{\color{blue}x}{\color{red}y} < 0\}$. Here the color-coded coordinates correspond to the local coordinates given by the bifoliation.

Furthermore, if $\pi\colon \widetilde{M} \to \mathcal{P}$ is the projection, then taking the derivative we get embeddings
$$\pi_{\pm} \colon \widetilde{M} \longrightarrow \mathbb{P} (T \mathcal{P}) \ , \ p \longmapsto d\pi (\xi_{\pm}).$$
In particular, $\xi_{\pm} = \pi_{\pm}^*\xi_{can}$ are given via pulling back the canonical contact structure on the projectivised tangent bundle, which can be identified with the unit tangent bundle after choosing a metric. 

 This then gives an alternate proof of the well known fact going back to the work of Eliashberg-Thurston \cite{ETh} that the contact structures $\xi_\pm$ are universally tight, purely in terms of the bifoliated plane. They are furthermore symplectically fillable and hence have no Giroux torsion (cf.\ \cite{BBM}):
\begin{theorem}[Anosov implies universally tight]\label{thm:Anosov_tight}
    The contact structures $\xi_\pm$ associated to an Anosov flow are universally tight and strongly symplectically fillable. In particular, they have no Giroux torsion.
\end{theorem}

\subsection{Orbit Equivalence and Contact Structures} 
The natural notion of equivalence of Anosov flows is {\it orbit equivalence} where two flows are equivalent if there is a homeomorphism preserving flow lines. We say that the flows are  related by an {\it isotopy orbit equivalence}, and are {\it isotopy orbit equivalent}, if the orbit equivalence is isotopic to the identity. It is not {\em a priori} clear how such orbit equivalences act on the associated contact structures. This was however recently resolved and we have
\begin{theorem}[\cite{BoM} Theorem D] \label{thmintro:anosovbicontact}
    Let $X_0$ and $X_1$ be two Anosov flows on $M$ supported by bi-contact structures $(\xi^0_-, \xi^0_+)$ and $(\xi^1_-, \xi^1_+)$, respectively. If $\Phi_0$ and $\Phi_1$ are orbit equivalent, then $(\xi^0_-, \xi^0_+)$ and $(\xi^1_-, \xi^1_+)$ are deformation equivalent through bi-contact structures. 
    
    More precisely, if $h \colon  M \rightarrow M$ is an (oriented) orbit equivalence between $X_0$ and $X_1$, then $h$ is isotopic to a smooth diffeomorphism $\widetilde{h} : M \rightarrow M$ such that $\big(\widetilde{h}_*(\xi^0_-), \widetilde{h}_*(\xi^0_+)\big)$ and $(\xi^1_-, \xi^1_+)$ are homotopic through bi-contact structures.
\end{theorem}
We then have the following important consequence:
\begin{corollary}[Invariance of Contact Structures]
    If two Anosov flows $X,X'$ are orbit equivalent then the contact structure $\xi_\pm$ and $\xi'_\pm$ are contactomorphic.
\end{corollary}

A key characterization of Anosov flows up to orbit equivalence that we will use is due to Barthelm\'e, Frankel and Mann.

\begin{theorem}\label{theorem: BFM}\cite[Theorem 1.1]{BFM} On a hyperbolic closed $3$-manifold, an Anosov flow is completely determined, up to isotopy orbit equivalence, by the set of non oriented free homotopy classes of its periodic orbits.
\end{theorem}

This statement takes into account the fact that on a hyperbolic manifold, there are no {\it tree of scalloped regions}, see \cite{BFM}, and that every Anosov flow on any closed atoroidal manifold is transitive by a result of Brunella \cite{Br}. The Theorem also remains true for transitive flows if all JSJ-pieces are hyperbolic or if all Seifert pieces are {\bf free} in the sense of Barbot-Fenley as well as for pseudo-Anosov flows. 
\begin{remark}[Finite Index Subgroups and Orientations]\label{rem:Finite_Index}
    Let $\gamma\in \pi_1(M)$ be primitive element. A key observation is that a power $\gamma^k \in \pi_1(M)$ is represented by a periodic orbit (possibly non simple) if and only if $\gamma$ or $\gamma^{-1}$ is. This is a manifestation of the fact that $\gamma$ is periodic if and only if its action on $\mathcal{P}$ has a fixed point. In particular, this means that the periodic orbit data of a flow is uniquely determined by that on any finite index subgroup $\Gamma \subseteq \pi_1(M)$. 

    Thus by taking an appropriate cover (of fixed degree) corresponding the kernel of the natural projection $\pi_1(M) \to H_1(M,\Z_2)$, we can assume that all stable/unstable bundles are orientable and that the flow respects the orientations.
\end{remark}
\subsection{Dynamics of Elements on $\mathcal{P}$}
The bifoliated plane associated to any (pseudo)-Anosov flow can be compactified to a disk $\mathcal{P} \cup \partial \mathcal{P}$ in a natural way, so that the action of $\pi_1(M)$ extends. One then has the following trichotomy for how elements can act. For cleanliness we assume that all actions are orientation preserving:
\begin{theorem}[Proposition 3.6 \cite{BFM}]\label{thm:dynamics_bi_plane}
    Let $\gamma \in \pi_1(M)$. Then precisely one of the following holds:
    \begin{itemize}
        \item[(i)] (Fixed Points) There is a fixed point in $\mathcal{P}$ and the element $\gamma$ is said to be periodic.
        
        \item[(ii)] (No Fixed Points: Loxodromic) There are no fixed points in $\mathcal{P}$ and precisely two on $\partial \mathcal{P}$ and the action is (topologically) conjugate to a loxodromic M\"obius transformation in $PSL(2,\R)$. 
        \item[(ii)'] (No Fixed Points: Parabolic) There are no fixed points in $\mathcal{P}$ and precisely one on $\partial \mathcal{P}$ and the action is (topologically) conjugate to a parabolic M\"obius transformation in $PSL(2,\R)$.

        \item[(iii)] (No Fixed Points: Scalloped) There are no fixed points in $\mathcal{P}$, $\gamma$ preserves a scalloped region and the number of fixed points on $\partial \mathcal{P}$ is 4.
    \end{itemize}
    In particular, in cases (ii)' and (iii) there are infinite sequences of non-separated leaves which cannot occur if $M$ is atoroidal or if the flow is skew $\R$-covered. 
\end{theorem}
\begin{remark}[Scalloped regions and transverse tori]\label{rem:scallop_tori}
    The stabiliser of any scalloped region is virtually $\Z^2$ and hence gives rise to an incompressible torus in $M$, after possibly taking some finite covering. Moreover, this torus can be made embedded and transverse to the flow by a result of Barbot (Th\'{e}or\`{e}me B, \cite{Bar2}). In particular, if there are no transverse tori then the scalloped case is automatically excluded.
\end{remark}
 
\begin{remark}[Centralisers]\label{rem:centraliser}
      Case (iii) above can occur only if the centraliser $Z(\gamma)$ is not virtually cyclic, i.e.\ it contains a copy of $\Z^2$. Furthermore, if $Z(\gamma)$ is not virtually cyclic, then all elements have to act on the bifoliated plane with the same dynamics and we have a dichotomy: It is known that any $\Z^2 \subseteq \pi_1(M)$ cannot act freely. If there is a single (maximal) cyclic subgroup that has fixed points we must be in case (ii) or (ii)'. If this is not the case, then there are exactly two and we are in case (iii) by (Corollary 4.24 \cite{BaM}). Here we use the fact that commuting elements must have the same number of fixed points on the boundary.
\end{remark}
Furthermore, given any incompressible torus $T$, Barbot's result shows that it is scalloped if and only if there are (primitive) elements $c_1,c_2 \in \pi_1(T) \cong \Z^2$ that are homotopic to periodic orbits and generate a finite index subgroup. Moreover, up to taking inverses there are {\em precisely} two such elements.

\section{Slope functions}

Let $(M,\xi)$ be a universally tight contact irreducible $3$-manifold so that the pullback $(\widetilde{M}, \widetilde{\xi}) \cong (\R^3,\xi_{st})$ is the (standard) tight contact structure.

\subsection{$\Z$-covers} If $\gamma \in \pi_1(M)$ is non-trivial, we consider the infinite
cover associated with $\gamma$
$$(S^1\times \R^2,\xi^\gamma) \cong (\widetilde{M}/\langle \gamma \rangle, \xi^\gamma) \longrightarrow (M,\xi).$$
Here $\xi^\gamma$ denotes the pulled back contact structure. We set $M^\gamma:=\widetilde{M}/\langle \gamma \rangle$.


\subsection{Asymptotic Slopes} Given an exhaustion of $(M^\gamma,\xi^\gamma)$ by concentric solid tori $T_n$, whose boundary $2$-tori $\partial T_n$ are convex, we consider the sequence of dividing curves $\Gamma_{\partial T_n}$ on  $\partial T_n$. Since $\xi^\gamma$ is tight, $\Gamma_{\partial T_n}$ consists of parallel essential curves in $\partial T_n$. 
If we fix  a common longitude $l$ for the tori $\partial T_n$, the direction of every dividing set $\Gamma_{\partial T_n}$ can be detected by an angle $\theta_n(\xi^\gamma,l) \in (-\pi/2, \pi/2)$. Here $0$ stands for the direction of the longitude and $\pm \pi/2$ for the direction of the meridian; the latter meridian slopes are not reached since the structure is tight. 
Note that we take the basis given by $(longitude,meridian)$, which is the reverse of the usual $(meridian,longitude)$ one.
We consider the associated sequence of {\it slopes}, defined as
$$s_n(\xi^\gamma,l):= \tan(\theta_n(\xi^\gamma,l)).$$

By tightness, this sequence is monotone, 
non-increasing if $\xi$ is positive and non-decreasing if $\xi$ is negative, see \cite{Gi1}, so that it has a limit in $\R\cup\{\pm\infty\}$ as $n\to \infty$; alternately see Makar-Limanov \cite{M-L}.  Moreover, again by monotonicity of the slopes for concentric tori, this limit does not depend on the exhausting sequence.
We call it the {\bf asymptotic slope of $\xi^\gamma$} and denote it $Sl_{\xi} (\gamma,l)$, where $l$ recalls the choice of a longitude. We will suppress $l$ from the notation whenever the choice is understood.

We will say that $\xi^\gamma$ has {\bf constant slope at infinity} if there is an exhausting sequence of tori which have constant slope.
Now, if $\xi$ and $\xi'$ are two universally tight contact structures on $M$, we notice that the identity $Sl_{\xi} (\gamma,l)=Sl_{\xi'}(\gamma, l)$ does not depend on the choice of the longitude $l$, hence the intrinsic notion of having the same slope.

 \begin{remark}
     Note that if $\xi$ was (virtually) overtwisted the sequence of slopes would have a priori no limit as the dividing set might keep on rotating indefinitely.
 \end{remark}

 Elaborating a little more, we can also consider the asymptotic slope modulo the integers to get a quantity that does no longer depend on the choice of $l$, that 
we denote $\overline{Sl}_\xi (\gamma)$. We thus get a {\bf slope map} $$\overline{Sl}_\xi :\pi_1 (M) \to \R/\Z \cup \infty.$$ 
It depends on $\gamma$ only up to free homotopy. This map satisfies a homogeneity property on cyclic subgroups: We have $\overline{Sl}_\xi(\gamma^n)=  \vert n\vert \overline{Sl}_\xi(\gamma)$ for any integer $n\in \Z$.

\begin{remark}\label{remark: relative} An even more interesting way to get rid of the choice of a longitude is when we have a pair of universally tight structures. In that case, we can take the limit slope of one, provided it is not $\pm\infty$, as a reference for computing the other. The result is denoted $Sl_{\xi_1} (\gamma,\xi_2)$, when $\xi_2$ is our reference. It would be interesting to investigate the properties of this {\bf relative slope map} 
$$Sl_{\xi_1} (.,\xi_2): \pi_1(M) \to\R\cup\{\pm \infty\}$$
when $\xi_1$ and $\xi_2$ are two universally tight contact structures (e.g. a positive and negative deformation of a taut foliation). For example one has again $Sl_{\xi_1} (\gamma^n,\xi_2)=\vert n\vert Sl_{\xi_1} (\gamma,\xi_2)$.
\end{remark}

We now give several basic examples.
\begin{example}[Standard contact tubes]
    Let $\xi = \mathrm{Ker}(d\theta + r^2d\varphi)$ on $S^1_\theta \times \R^2$ where we take polar coordinates $(r,\varphi)$ on $\R^2$.
     The asymptotic slope on the open tube $S^1 \times \mathring{D}_{R}$ of radius $R = r^2$ is then $-\frac{\partial \varphi}{\partial \theta}= \frac{1}{R}$. 
\end{example}
\begin{example}[Canonical contact tubes]\label{ex:can_con_tube}
    Let $\xi_{can} = \mathrm{Ker}(\cos(\theta)dx - \sin(\theta)dy)$ be the canonical contact structure on $U^* \R^2 \cong S^1_\theta \times \R^2$.  Then the characteristic foliation on $\partial (S^1 \times D_{R})$ is independent of $R$, as one just rescales the contact form by dilating in the $(x,y)$-plane. Moreover, we parametrise $\partial (S^1 \times D_{R})$ by $$(\theta, \varphi) \longmapsto (\theta, \cos(\varphi),\sin(\varphi))$$
    so that the characteristic foliation is given by 
   $$(\cos(\theta)\sin(\varphi) + \sin(\theta)\cos(\varphi))d\varphi = 0$$
    This is then ruled by vertical $S^1_\theta$-fibers and has tangency locus  along the Legendrian curves given in $(\theta,\varphi)$-coordinates by 
    $$L_0 = \{\theta = -\varphi \} \ , \ L_\pi = \{\theta = -\varphi + \pi \} .$$
    Thus the slope is $1$ for the longitude given by the fiber direction and for the open solid torus it is constant at infinity. Taking a $k$-fold cover in the $\theta$-factor gives tubes with slope $k$.
    
    Note that making a small perturbation given by pushing along a Legendrian vector field $X= \sin(\theta)\partial_x + \cos(\theta)\partial y$, we see that $\xi_{can}$ is isotopic to a (positively) transverse contact structure $\xi_{tr}$. Considering the contact structure as an Ehresmann connection, the holonomies along the boundary of a disk $D_R$ are monotonic. Let $T_n = S^1 \times D_{R_n}$ be the solid torus of radius $R_n$ and 
    let $l_{fib}$ be the longitude given by the fibers. Using our orientation conventions we get 
    $$s_n(\xi_{tr},l_{fib}) > 0. $$
For an oppositely oriented contact structure $\xi^-_{tr}$ -- given by say reflecting in $\theta$ -- we then have (again on $S^1 \times D_R$):
$$s_n(\xi^-_{tr},l_{fib}) < 0. $$
\end{example}

\begin{example}[The $3$-torus case] For $n\in \N^*$, consider the contact structures $\xi_n^\pm$ of equation $\cos(nt)dx\pm \sin(nt)dy=0$ on $T^3=(\R/\Z)^3_{(x,y,t)}$.
We let $\gamma_x :=\{ y=t=0\}$ and $\gamma_t :=\{ x=y=0\}$.
One can check that $Sl_{\xi_1^-} (\gamma_x, \xi_n^+)=0$, but $Sl_{\xi_1^-} (\gamma_t,\xi_n^+)=n+1$, that is $Sl_{\xi_1^-}$ distinguishes between the structures $\xi_n^+$, at least up to isotopy, whereas they are all deformations
of the same foliation $\{ t=C\}$.
\end{example}

Note that if a tight contact structure on an open solid torus extends to a $C^1$-plane field on the closed torus so that the characteristic foliation on the boundary is stable (e.g.\ Morse-Smale), then its slope is constant at infinity. The converse may however not hold: the number of dividing curves may be unbounded and the contact structure, despite having constant at infinity slope, may not extend to a solid torus.

\section{Detecting Periodic Orbits via Contact Structures}

Let $X$ be an Anosov flow with associated bi-contact pair $(\xi_+,\xi_-)$. We prove that the pair $\xi_\pm^\gamma$ of pulled-back contact structures on the infinite cover associated with $\gamma$ distinguishes between the cases when (the free homotopy class of) $\gamma^{\pm 1}$ can be realized by a periodic orbit of $X$ or not.\footnote{Note this is equivalent to the fact that $\gamma$ acts on the bifoliated plane coming from the flow with fixed points or not.} Our goal is thus to prove the following:


\begin{theorem}\label{thm: homotopy} Let $X,X'$ be Anosov flows on a closed $3$-manifold $M$ and $\gamma$ a free homotopy class of loops. Assume that 
the contact structures $\xi_\pm$ and $\xi'_\pm$ associated with $X$ and $X'$ are pairwise isotopic. Then $\gamma$ or $\gamma^{-1}$ is represented by a periodic orbit for $X$ if and only if it is for $X'$.
\end{theorem}
\begin{remark} We are assuming the stable/unstable foliations are oriented. If not, we go to the orientation cover and use the fact that $\gamma^n$ is an orbit if and only if $\gamma$ is to deduce finiteness in general.
\end{remark}
As a direct consequence we obtain Theorem \ref{theorem: classification} from the introduction, which we restate for convenience.
\vspace{0.2cm}

\begin{thm} {\rm 1.2.} On a closed hyperbolic $3$-manifold $M$, the following are equivalent:
\begin{enumerate} 
    \item[(a)] two Anosov flows $X$ and $X'$ are isotopy orbit equivalent;
    \item[(b)] the associated contact structures $\xi_+$, $\xi_+'$ and $\xi_-$, $\xi_-'$ are pairwise isotopic;
    \item[(c)]  the contact pairs $(\xi_+,\xi_-)$ and $(\xi'_+,\xi'_-)$ are deformation equivalent among transverse contact pairs via a map that is isotopic to the identity;
    \item[(d)] $X$ and $X'$ are homotopic through projectively Anosov flows.
\end{enumerate}
\end{thm}
\begin{proof}[Proof of Theorem \ref{theorem: classification}] The equivalence $(a) \Longleftrightarrow (b)$ follows from Theorem \ref{thm: homotopy} and 
Theorem \ref{theorem: BFM}. The implication $(c) \Longrightarrow (b)$ is clear. Theorem D of \cite{BoM} yields $(a) \Longrightarrow (c)$ and $(c) \Longleftrightarrow (d)$ is clear as the projective Anosov condition is equivalent to a bi-contact structure by \cite{ETh}, resp.\ \cite{Mit}.  
\end{proof}

\subsection{General Periodic Case}
The periodic case requires no extra hypothesis on $M$.

\begin{lemma}\label{lemma: slope0} If $X$ is an Anosov flow on a closed $3$-manifold $M$ and if $\gamma \in \pi_1(M)$ can be represented by a periodic orbit of $X$, then the contact structures $\xi_+^\gamma$ and $\xi_-^\gamma$ have equal constant slope at infinity.
\end{lemma}
\begin{proof}
Again, if $X$ is an Anosov flow in $M$, with associated contact structures $\xi_\pm$, and $\gamma$ is a periodic orbit of $X$, then the infinite cover of M given by $\gamma$ is diffeomorphic to the open solid torus $M^\gamma \cong S^1 \times \R^2$.
It contains a properly embedded infinite cylinder $C_\gamma$ which is the pull-back of the unstable leaf of $\gamma$. 

The cylinder $C_\gamma$ is foliated by 1-dimensional leaves of 
the pull-back $X^\gamma$ of $X$ and the foliation is homeomorphic to the geodesic foliation. Namely it contains a central periodic orbit $\overline{\gamma}$ pulling-back $\gamma$, and then non compact leaves all asymptotic to $\overline{\gamma}$ at $-\infty$ (there are no “generalized Reeb components”). This foliation of $C_\gamma$ also coincides with its characteristic foliation for both $\xi_+^\gamma$ and $\xi_-^\gamma$.

One can then find an exhaustion of $M^\gamma$ by concentric solid tori $S^1 \times D_n$ that intersect $C_\gamma$ along compact annuli containing $\overline{\gamma}$ and such that $X^\gamma$ is positively transverse to $T_n =\partial (S^1 \times D_n)$ along $C_\gamma\cap T_n$. This implies that, assuming $T_n$ is convex, the dividing set of $T_n$ for both $\xi^\gamma_\pm$ is parallel to $C_\gamma\cap T_n$. Indeed, one can cut $S^1\times D_n$ in two parts along $C_\gamma$. Each part is a solid torus (with corners), whose boundary  $2$-torus (with corners) $T_n^1$, resp. $T_n^2$, contains the annulus $C_n:=C_\gamma\cap (S^1\times D_n)$. The characteristic foliations of this annulus $C_n$ for both $\xi_\pm^\gamma$ is the standard one described above generated by $X^\gamma$: $\overline{\gamma}$ is a non singular periodic orbit,  repelling for one torus, say $T_n^1$, attracting for the other $T_n^2$ (since the coorientation of $C_n$ changes depending on the $2$-torus $T_n^i$ we consider) and it is exiting, resp. entering, along both of its boundary components. 

By tightness of $\xi_\pm^\gamma$, we know that the dividing curves on each of the two $2$-tori $T_n^i$ are essential and thus parallel. Moreover $\overline{\gamma}$ is contained in one of their  complementary regions  and thus they have to be parallel to $\overline{\gamma}$. Now, a small neighbourhood of the attracting locus of the characteristic foliation of $T_n^1$, resp. repelling locus for $T_n^2$, (made, in the generic case where it is Morse-Smale \cite{Gi1}, of the union of negative singularities together with their unstable separatrices and attracting periodic orbits; resp. positive singularities together with their stable separatrices and repelling periodic orbits) is by tightness a non empty union of small annuli contained in  $T_n^1 \setminus C_n$, resp. $T_n^2 \setminus C_n$. These neighbourhoods are thus also contained in $T_n$ where they determine the direction of the dividing set, given by their boundary and constantly  equal to the direction of $C_\gamma \cap T_n$.

Thus $\xi_\pm^\gamma$ have constant and equal slopes at infinity. Note that by monotonicity, they are then constant for any large enough exhaustion of tori.
\end{proof}

\subsection{Skew $\R$-covered Case}  We can now distinguish the case when $\gamma$ is not periodic and the orbit space is skew. This then gives a contact homology free proof of the main finiteness result of Bowden-Barthelm\'{e}-Mann \cite[Theorem 1.11]{BBM} in the hyperbolic case.

\begin{lemma}\label{lem:R_covered_periodic} If $X$ is an Anosov flow on a closed $3$-manifold $M$ which is $\R$-covered and skew, and if $\gamma \in \pi_1(M)$ is not represented by a periodic orbit of $X$, then
$Sl_{\xi_+}(\gamma,l)>Sl_{\xi_-}(\gamma,l)$ and the slope of any $2$-torus isotopic to a concentric one in $M^\gamma$ for $\xi_+^\gamma$ is different from the slope of any other one for $\xi_-^\gamma$.
\end{lemma}
\begin{proof}
We consider the orbit space. The skew property means that it is diffeomorphic to a diagonal strip of finite width in the plane, where the stable and unstable foliations project to the foliations by horizontal and vertical lines.
The action of $\gamma$ on the orbit space takes one horizontal line to a different one, meaning that the quotient is obtained by identifying these two boundary horizontal lines of 
the fundamental domain they bound. This gives an open annulus $A$. To obtain the quotient $(M^\gamma,\xi_\pm^\gamma)$, it remains to thicken $A$ in the direction of $X$, which gives a product structure $A\times \R$.
We consider a sequence of tori obtained by first considering an exhaustion of $A$ by closed annuli $A_n$ and then a large enough thickening.

The sequence $A_n$ can be chosen so that, for $n$ large enough, the tangent lines $T\partial A_n$ (appropriately oriented) are both positively transverse to $\mathcal{F}^u$ and $\mathcal{F}^s$: one may think of the annulus as having boundary `parallel' to the dotted ideal boundary depicted in Figure \ref{fig:skew_plane_annulus}. As indicated above contact condition implies that on a given orbit of $X$, when $t\to +\infty$, $\xi_\pm \to \pm T\mathcal{F}^u$ under the flows and when $t\to -\infty$, $\xi_\pm \to T\mathcal{F}^s$.

 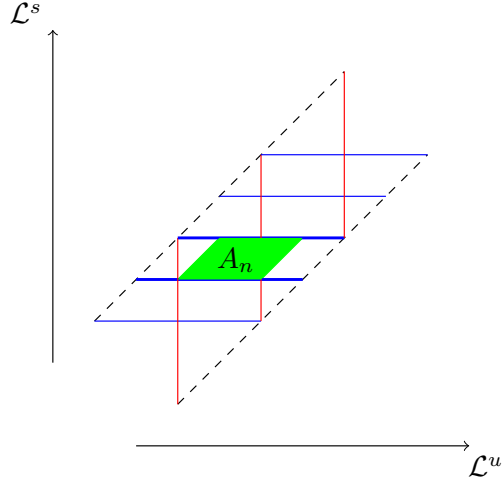
\begin{figure}[h]
\centering
\scalebox{1.1}{
\begin{tikzpicture}[x=1cm,y=1cm]

\draw[->,line width=0.4pt] (1,0) -- (5,0);
\draw[->,line width=0.4pt] (0,1) -- (0,5);

\draw[dashed,line width=0.4pt] (1.5,0.5) -- (4.5,3.5);
\draw[dashed,line width=0.4pt] (0.5,1.5) -- (3.5,4.5);

\draw[green,line width=0.4pt] (1.5,2) -- (2,2.5);

\draw[green,line width=0.4pt] (2.5,2) -- (3,2.5);

\draw[blue,line width=0.4pt] (0.5,1.5) -- (2.5,1.5);

\draw[blue,line width=1pt] (1,2) -- (3,2);

\draw[red,line width=0.4pt] (1.5,0.5) -- (1.5,2.5);

\draw[blue,line width=1pt] (1.5,2.5) -- (3.5,2.5);

\draw[blue,line width=0.4pt] (2,3) -- (4,3);

\draw[red,line width=0.4pt] (2.5,1.5) -- (2.5,3.5);

\draw[blue,line width=0.4pt] (2.5,3.5) -- (4.5,3.5);

\draw[red,line width=0.4pt] (3.5,2.5) -- (3.5,4.5);

\node[below] at (5.2,0) {$\mathcal{L}^u$};
\node[left] at (0,5.2) {$\mathcal{L}^s$};

\fill[green, opacity=0.25]
    (1.5,2) --
    (2,2.5) --
    (3,2.5) --
    (2.5,2) --
    cycle;
    
\node[below] at (2.2,2.5) {$A_n$};

\end{tikzpicture}
}
\caption{The shaded region corresponds to the compact annulus $A_n$ in the skew-plane $\mathcal{P}_\R$ where the thick blue leaves are identified under $\gamma$.}
\label{fig:skew_plane_annulus}
\end{figure}

We consider the contact structure corresponding to the positive quadrants on $\mathcal{P}_{\R}$, which in the positively skew plane is $\xi_-$. Then taking a large enough portion $[-C_n,C_n] \subset \R$ of the $X$ orbits, we obtain the torus $T_n =A_n \times [-C_n,C_n]$ with piecewise smooth boundary. If $C_n$ is large enough, we can make sure that for every $p\in \partial A_n$, there exists $t(p)\in (-C_n,C_n)$ such that 
$\xi_-(p,t(p)) = T_{(p,t(p))}(\partial A_n\times [-C_n,C_n])$. This means that the $\xi_-$ characteristic foliation of the torus $\partial T_n$ has a circle of singularities along the sides
$\partial A_n \times [-C_n,C_n]$. In particular the $\xi_-$ slope is constant for $n$ large enough and equal to the one of $\partial A_n \times \{0\}$.

We now observe that the contact structure $\xi_+$ is transverse to $\partial T_n$. Its characteristic foliation is tangent to $X$ on the sides $\partial A_n \times [-C_n,C_n]$
and contained in the sector between $\mathcal{F}^u$ and $\mathcal{F}^s$ that avoids $\xi_-$ on the other sides $A_n\times \{\pm C_n\}$. This means that on the boundary of the universal cover of $T_n$, it intersects the line of singularities of $\xi_-$ infinitely many times with a constant sign: the slope of the dividing set of $\partial T_n$ for $\xi_-^\gamma$ is different from that of $\xi_+^\gamma$. 
We then get that $Sl_{\xi_+}(\gamma,l)>Sl_{\xi_-}(\gamma,l)$, the slope for $\xi_+^\gamma$ being obtained, as the limit of the ones on $\partial T_n$, by the slope on the compactification of $A\times \R$, where we follow alternatively the stable and unstable leaves on the back and on the front annuli. 
Last, by monotonicity, the slope of every (isotopic to a) concentric $2$-torus for $\xi_+^\gamma$ is larger than the limit slope for $\xi_-^\gamma$ and thus than that of every (up to isotopy) concentric $2$-torus slope for $\xi_-^\gamma$.
\end{proof}

\begin{remark}
    Note that the result above does not, and indeed cannot, distinguish between $\gamma$ and $\gamma^{-1}$ and it is fortunate that Theorem \ref{theorem: BFM} result does not need this.
\end{remark}

\subsection{Non-$\R$-covered Case} 
\begin{center}
\textbf{Assumption:} We shall assume that $M$ is hyperbolic. 
\end{center}





\noindent Here our strategy is to use the special structure of the action of elements that act without fixed points on the bifoliated plane of an Anosov flow to embed the contact cover $(M^\gamma,\xi_\pm^\gamma)$ in standard examples where one can show that the slopes do not agree. Note that the assumption that $M$ is hyperbolic is not strictly necessary in what follows as it is only used to exclude scalloped regions in the bifoliated plane associated to the flow (cf.\ Theorem \ref{theorem: BFM}).

\subsection*{Rotation Number for (Bi)Foliations} When $\gamma$ acts on the bifoliated plane $\mathcal{P}$, taking $x\in \mathcal{P}$, we consider an embedded arc $\alpha \subset \mathcal{P}$ joining $x$ to $\gamma(x)$ and that is tangent to  the stable foliation line field and agreeing with the orientation at $x$ and $\gamma(x)$. 

\begin{definition}
Let $rot_s(\alpha) = rot(\dot{\alpha},\mathcal{F}^s)$ be the rotation number of the (oriented) tangent line field of the stable foliation relative to the tangents of $\alpha$.
\end{definition} 
\noindent Note one can also define a similar number for the unstable foliation, but as the foliations are transverse these would then agree.

If $\gamma$ has no fixed point, the infinite cover of $M$ associated with $\gamma$  is diffeomorphic to $A \times (0,1)_t$, where $A$ is an annulus obtained as quotient of the bifoliated plane $\mathcal{P}$ and $(0,1)_t$ is directed by $X$. This follows from Theorem \ref{thm:dynamics_bi_plane}, since we are assuming that the manifold is hyperbolic. The rotation number of $\mathcal{F}^s$ along $\gamma$ is then given by the rotation of the stable line field along the core curve of $A$. In this case the rotation number of the stable foliation with respect to $\dot{\alpha}$ along $\alpha$ is an integer that is independent of $x$, since any two core curves are isotopic.

\begin{lemma} Let $\gamma \in \pi_1(M)$ be an element acting freely on the bifoliated plane $\mathcal{P}$ so that the quotient is a cylinder. Then the rotation number $rot_s(\gamma) \in \Z$ is well defined.
\end{lemma}

\subsection*{Extending Bifoliated Annuli}
\begin{lemma}\label{lemma: embedding} Let $1 \ne \gamma \in \pi_1(M)$ be a non-periodic class and let $A$ be the annulus, quotient of the bifoliated plane $\mathcal{P}$ by the action of $\gamma$. We consider a compact essential annulus $A_0$ in $A$. 
If $rot_s(\gamma)$ is zero, then there is an embedding of $(A_0,\mathcal{F}^s, \mathcal{F}^u)$ into a skew bifoliated annulus $A_1$.
\end{lemma}
\begin{proof} 

 We first prove that, maybe after replacing the annulus $A_0$ with a nearby one, there are two leaves
$\mathcal{C}^s$ of $\mathcal{F}^s \cap A_0$ and $\mathcal{C}^u$ of $\mathcal{F}^u \cap A_0$ that intersect each other and both cross $A_0$ from one boundary component to the other, see Figure \ref{fig:placeholder}.

\begin{figure}[h]
    \centering
\begin{tikzpicture}[scale=1.4, line width=1pt]

\draw[black,dashed] (-3,1) -- (3,1);
\draw[black,dashed] (-3,-1) -- (3,-1);

\draw[->,blue] (0,-1) -- (0,0.5);
\draw[blue] (0,0.1) -- (0,1);

\node[below] at (0.1,-1) {{\color{blue} $x^s$}};
\node[below] at (0.1,1.6) {{\color{blue} $y^s$}};

\draw[red] (-2.3,-1) -- (-2.3,0.5);
\draw[<-,red] (-2.3,0.1) -- (-2.3,1);

\draw[->,red]
(-2,-1)
.. controls (-1,1.6) and (1,-1.6) ..
(2,1);

\draw[->,red]
(-0.5,1)
.. controls (0,0.5) ..
(0.5,1);

\draw[->,red]
(-1,1)
.. controls (0,0) ..
(1,1);

\node[right] at (3,1) {$\partial A_0$};

\node[above] at (2.1,1) {{\color{red} $y^u$}};

\node[above] at (-1.9,-1.5) {{\color{red} $x^u$}};

\end{tikzpicture}
    \caption{Crossing leaves, where the dashed lines represent the boundary components of $A_0$.}
    \label{fig:placeholder}
\end{figure}
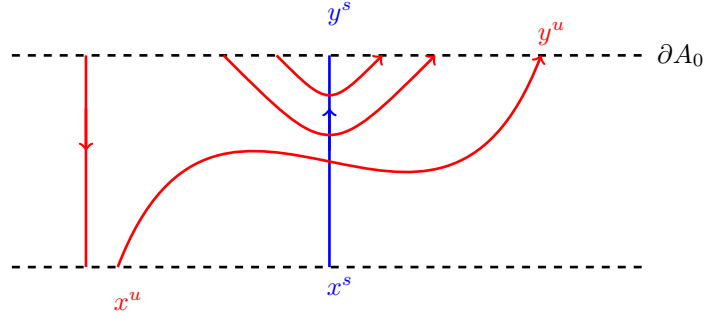

Indeed, there exists a regular point of the axis of translation of $\gamma$ for the stable line field in $\mathcal{P}$: that is takes leaves $\ell_s$ of $\mathcal{F}^s$ so that $\ell_s$ separates $\gamma(\ell_s)$ and $\gamma^{-1}(\ell_s)$. Its existence is guaranteed by the existence of an axis in the leaf space of the foliations: see \cite{BaM} Section 3.3. In particular, viewed in $A$, the foliation $\mathcal{F}^s$ has no closed leaves, and same holds for $\mathcal{F}^u$ by the same property.
One component of $\ell_s \cap A_0$ gives $\mathcal{C}^s$.  After slightly expanding $A_0$ we can assume that $\mathcal{C}^s$ intersects the boundary transversely.

We now consider the leaves of $\mathcal{F}^u \cap A_0$ intersecting $\mathcal{C}^s$.
Since $\mathcal{F}^u$ has no closed leaves in $A_0$ (nor $A$), these leaves either cross $A_0$ or are parallel to one of the boundary components. We consider the set of arcs included in $\mathcal{F}^u \cap A_0$ whose boundary points are on the same component of $\partial A_0$, whose interiors lie in the interior of $A_0$ and that intersect $\mathcal{C}^s$.
 This set of arcs intersects $\mathcal{C}^s$ in two disjoint  (possibly empty) subsets determined by the boundary component of $\partial A_0$ on which the arcs end.  
 Now if there were no crossing leaf intersecting $\mathcal{C}^s$ (possibly tangent to $\partial A_0$ at some interior point of the leaf), both of these subsets of $\mathcal{C}^s$ would be non-empty as the two endpoints of $\mathcal{C}^s$ would be the boundary points of two arcs, possibly reduced to a point (e.g. when there is a quadratic contact of the corresponding  unstable leaf from the outside of $A_0$ at this point), in the two corresponding subsets of arcs. Under this hypothesis, they would also be closed subsets of $\mathcal{C}^s$, a contradiction to the connectedness of $\mathcal{C}^s$. This proves the existence of the intersecting crossing leaves $\mathcal{C}^s$ and $\mathcal{C}^u$.
 

 We denote the intersections of $\mathcal{C}^s$ and $\mathcal{C}^u$ with the the boundary of $A_0$ by, respectively, $x^s$, $y^s$ and $x^u$, $y^u$, the $x$'s and the $y$'s being contained in the same component of $\partial A_0$, and orient the foliations $\mathcal{F}^s$ and $\mathcal{F}^u$ so that  $\mathcal{C}^s$ and $\mathcal{C}^u$ are both going from the $x$'s to the $y$'s, see Figure \ref{fig:placeholder}.

 Now, we claim that, after maybe slightly modifying $\partial A_0$ to adjust some transversality signs and eliminate interior tangencies of the crossing leave $\mathcal{C}^u$ with $\partial A_0$, we obtain two crossing leaves so that along arcs in $\partial A_0$ joining $x^s$ to $x^u$ and $y^s$ to $y^u$, the degree of $\mathcal{F}^s$ and $\mathcal{F}^u$ is zero. Here, having degree zero means that there is a homotopy of the tangent line field to  $\mathcal{F}^s$ and $\mathcal{F}^u$ relative to $x^s$, $x^u$, $y^s$ and  $y^u$ to a line field that is transverse to $\partial A_0$.

 Indeed, we first deform a bit $A_0$ so that the crossing leaves $\mathcal{C}^s$ and $\mathcal{C}^u$ are intersecting $\partial A_0$ transversely and only at their endpoints. Then we can find an embedded triangle $T$ in $A_0$ whose sides are on $C^s$, $C^u$ and the boundary component of $\partial A_0$. 
We adjust $\partial A_0$ so that $\mathcal{F}^s$ is respectively entering and exiting $A_0$ at the endpoints $x^u$ and $y^u$ of $C^u$, and similarly for
$\mathcal{F}^u$ at $x^s$ and $y^s$, see Figure \ref{fig: adjust}. 
Along the edge of the triangle contained in $C^u$, the foliation $\mathcal{F}^s$ is entering the triangle. Thus, by isotoping the edge of $T$ contained in $\partial A_0$ keeping tangents fixed at the end points to be close to the other two sides of the triangle $T$, we can make it transverse to $\mathcal{F}^s$. Pulling back under the isotopy  then gives a $C^1$-smooth $1$-parameter family of vector fields $X_t$ along $[y^s,y^u] \subset \partial A_0$ fixed at the end points so that $X_0= \mathcal{F}^s$ and  $X_1$ is transverse to the boundary. In particular, we can then homotope the vector field on the remaining arc to be transverse to the boundary relative to its end points as well.

\begin{figure}[ht]
\begin{overpic}
[scale=0.7]{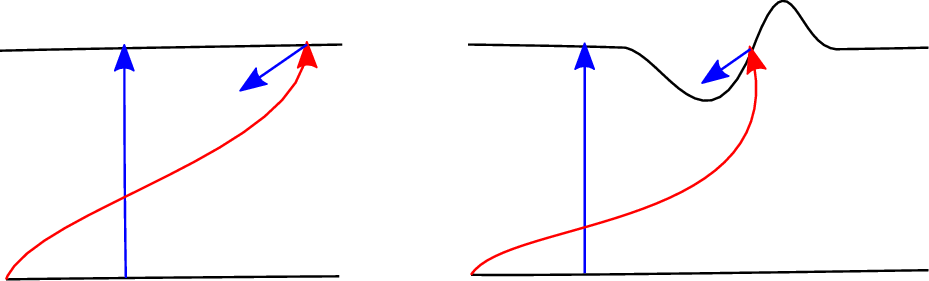}
\put(40,25){$\partial A_0$}
\put(77,27){\color{red}{$y^u$}}
\put(33,27){\color{red}{$y^u$}}
\put(13,27){\color{blue}{$y^s$}}
\put(62,27){\color{blue}{$y^s$}}

\end{overpic}
\caption{Example of a $C^0$-small modification of $\partial A_0$ to adjust the transversality sign of $\mathcal{F}^s$ at $y^u$ and match it with the one at $y^s$, without changing the one of $\mathcal{F}^u$.}
\label{fig: adjust}
\end{figure}
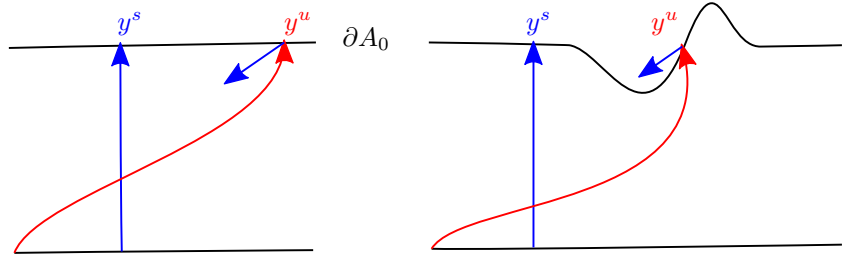

Once this is done, we can extend the tangent line fields of $\mathcal{F}^s$ and $\mathcal{F}^u$ to a larger annulus $A_1$, keeping them tangent to paths from the four points $x^s$, $x^u$, $y^s$ and  $y^u$ to the new boundary extending $\mathcal{C}^s$ and $\mathcal{C}^u$, so that we get two non singular transverse line fields on $A_1$, that are both transverse to $\partial A_1$, and each having a crossing leaf (the extensions of $\mathcal{C}^s$ and $\mathcal{C}^u$). These line fields have no integral periodic orbits because they have a crossing leaf (and by Poincar\'e-Bendixon also no contractible periodic orbit). This means that every leaf of both foliations crosses $A_1$.
Choosing the extension to agree with that coming from a skew-plane we get that $A_1$ then has the structure of a bifoliated skew annulus, i.e., it is diffeomorphic the quotient of $\mathcal{P}_{\R}$ by a fixed point free automorphism.
\end{proof}

\begin{corollary} In the non-periodic case, if the rotation along $\gamma$ is $0$, then the limit slopes of $\xi_+^\gamma$ and $\xi_-^\gamma$ are not constantly equal at infinity.
\end{corollary}
\begin{proof} Taking a large torus of the form $A_0 \times I \subset A\times \R$, by Lemma \ref{lemma: embedding} we can embed it in some $A_1\times \R$, in which we extend the contact structures $\xi_\pm^\gamma\vert_{A_0\times I}$ to $A_1 \times I$, being tangent to the $I$-factor and rotating between the extended stable and unstable manifolds. These extensions are still universally tight (since their universal covers can be embedded in the positive and negative standard tight $\R^3$, an exercise left to the reader) and, by Lemma \ref{lem:R_covered_periodic}, 
the slopes of $\xi_\pm^\gamma$ at the boundary of $A_0 \times I\subset A_1\times \R$ are different. Since the initial torus was arbitrary, it follows that the slopes cannot be both constant at infinity and equal.
\end{proof}
In general we also have to deal with the case that there is non-trivial rotation, in which case we extend the foliation to consist entirely of Reeb components.
\begin{lemma}\label{lemma: Reeb1} In the non-periodic case, if the rotation associated with $\gamma$ is $k\neq 0$ then, if $A_0$ is a compact essential annulus in $A$, there exists an embedding of 
$(A_0,\mathcal{F}^s, \mathcal{F}^u)$ in an infinite cylinder $A_1$ with a pair of transverse foliations each consisting of $2\vert k\vert$ Reeb-like crossing components.
\end{lemma}
\begin{proof} We can extend the tangent line fields, which are $C^1$, so that their angle map along the boundary of $A_1$ has exactly $2k$ tangencies and quadratic contact. That is the line field is generated by the vector field
$$\cos(n\theta)\partial_\theta + \sin(n\theta)\partial_y$$
outside some compact annulus in $S_\theta^1\times \R_y$. Note that since the rotation number is non-zero, the resulting foliations can have no periodic orbits, since then the rotation number must vanish.

In this case, the foliations are given (up to diffeomorphism) by $2\vert k\vert$ crossing Reeb-components. This follows by applying Poincar\'{e}-Bendixson, which implies that no orbit can visit an end more than once.
\end{proof}
As in the vanishing rotation case, we can extend the contact structures $\xi_\pm^\gamma$ from $A_0 \times I$ to $A_1\times I$ by taking plane fields tangent to the $I$-direction and that rotate between the stable and unstable directions. The extensions are still universally tight since again their universal cover embedds in the tight positive and negative $\R^3$, see also the proof of Lemma \ref{lemma: Reeb2} below.
\begin{lemma}\label{lemma: Reeb2} If $(A_1,\mathcal{F}^s, \mathcal{F}^u)$ is made of $2k$ Reeb-like crossing components, then the limit slopes of $\xi_\pm^\gamma$  on $A_1\times I$ are different. 
\end{lemma}
\begin{proof} 
We assume after applying a suitable diffeomorphism that the Reeb foliations are invariant under vertical translations of $A_1$. We then quotient out by an integer translation in the vertical direction to get a contact structure on $T^2 \times \R$ so that the characteristic foliations on each $T^2_s = T^2 \times \{s\} $ is given by Reeb components. This means that each $T^2_s$ is convex.  Applying Giroux flexibility, we can assume all tori are linearly ruled so that the contact structure is tangent to some $S^1$-direction in the tori. In particular, the contact structure is isomorphic to a ($k$-fold cover) of the canonical contact structure on the unit cotangent bundle of the cylinder $U^*A_1$. Moreover, unwrapping the cylinder we obtain $\xi_+^\gamma$ and $\xi_-^\gamma$ are related by a reflection in $s$. The contact structure $\xi_+^\gamma$ is then contactomorphic to a $k$-fold cover of $(U^*\mathbb{R}^2,\xi_{can})$ and the slope on a torus $U^* D_n$ with respect to the natural framing given by the $S^1$-action is $\frac{1}{k}$ (and $k$ in our framing whose longitude is the $S^1$ fiber), see Example \ref{ex:can_con_tube}. Hence reflecting we get $-\frac{1}{k}$ (resp. $-k$) and the slopes do not agree. 

 Alternately note that by Giroux Flexibility we can isotope the contact structures (independently) in $T^2 \times \R$ so that the characteristic foliations are supensions on each $T_t$ (which are fixed for both contact structures). In particular the contact structures are (up to isotopy) transverse to a common $S^1$-fibration after lifting to the $\Z$-cover $M^\gamma$ and hence the slopes are different (cf.\ Example \ref{ex:can_con_tube}).
\end{proof}
Thus we get, as in the rotation $0$ case:

\begin{corollary} In the non-periodic case, if the rotation $k$ along $\gamma$ is non-zero, then the slopes of $\xi_+^\gamma$ and $\xi_-^\gamma$ are not constantly equal at infinity.
\end{corollary}
\begin{proof}
     Every solid torus $T_n$ in $M^\gamma$ is contained is some $(A_0 \times I,\xi_\pm^\gamma)$ and thus embeds in its extension $(A_1\times I,\xi_\pm^\gamma)$.
    By monotonicity of the slopes, in view of Lemma \ref{lemma: Reeb2}, the slopes of $\partial T_n$ verify $s_n(\xi_-^\gamma,l)\leq -\vert k\vert$ and $s_n(\xi_+^\gamma,l)\geq \vert k\vert$. Thus the conclusion.
    \end{proof}
\subsection*{Conclusion of proof of Theorem \ref{thm: homotopy}}

 If $\gamma$ is periodic, the argument of Lemma \ref{lemma: slope0} holds without any assumption on $M$. Thus we deduce that the slopes $Sl_{\xi_\pm}(\gamma)$ are both equal and constant at infinity. 
 
 In the non-periodic case we still have to exclude the case of asymptotically stable equal slopes. To this end we have to consider the loxodromic and parabolic cases (ii)/(ii)' and scalloped case (iii) of Theorem \ref{thm:dynamics_bi_plane}. In the loxodromic/parabolic case the quotient of the bifoliated plane is a cylinder having two intersecting crossing leaves and we can again argue precisely as in the hyperbolic case. Note that the remaining scalloped case implies the existence of an essential transverse torus, which is excluded by assumption (cf.\ Remark \ref{rem:scallop_tori}).
 
 As noted in Remark \ref{rem:centraliser}, the scalloped case can occur only if the centraliser $Z(\gamma)$ contains a copy of $\Z^2$. If the dynamics of $Z(\gamma)$ are loxodromic or parabolic then at most one (maximal) cyclic subgroup can be periodic and the contact structures $\xi_{\pm}$ detect this as explained above: An element $\gamma'$ in $Z(\gamma)$ is not periodic if and only if the slopes of $\xi^{\gamma'}_+$ are not equal and constant at infinity. In the case that $\gamma$ is scalloped, there are {\em at least} two independent elements in $Z(\gamma)$ that are periodic and thus satisfy that $Sl_{\xi_\pm}(\gamma)$ are both equal and constant at infinity as Lemma \ref{lemma: slope0} holds without any assumptions. This in particular, already means that the pair of contact structures $\xi_\pm$ detects if the subgroup (and hence the element) is scalloped or not.

Now let $\Z^2 \cong A \subseteq Z(\gamma)$ be scalloped and consider the cover $M^A \stackrel{\pi_A}{\longrightarrow} M$ corresponding to this group. This is diffeomorphic to $T^2 \times \R$ and we can assume that the torus $T_0 = T^2 \times \{0\}$ is transverse to the flow. We let $\gamma_-,\gamma_+$ be two periodic orbits coming from opposite (say stable) sides of the scalloped region. Then the unstable cylinders $C_{\gamma_\pm}$ lift (as cylinders) to $M^A$ and are disjoint from $T_0$. 

Exactly as in the proof of Lemma \ref{lemma: slope0} we can find a sequence of (disjoint) tori $T^\pm_n$ that escape out the positive resp.\ negative ends so that each $T^\pm_n$ contains a compact annulus $A_n^\pm$ so that $\partial A_n^\pm \subseteq C_{\gamma_\pm}$ is transverse to the flow and which cobounds a solid torus with $C_{\gamma_\pm}$ that contains the periodic orbit of $C_{\gamma_\pm}$ in its boundary. This implies that (after small perturbation) that the slopes on all $T^\pm_n$ agree with those given by the periodic orbits and are hence constant. Since the contact structures have vanishing Giroux torsion (as do all finite covers), the contact structures must be non-rotative between any two consecutive tori for all but two of them. Thus he Giroux Normal Form of the pull-backs $\xi^A_\pm = \pi^*_A\xi_\pm$ have constant slopes on the ends that exactly detects which elements in $A$ are periodic.

Thus in all cases the contact structures determine which classes are periodic, completing the proof.

\section{Proof of main results}
\subsection{Finiteness of Anosov Flows on Closed Hyperbolic $3$-Manifolds} We can now prove the first main result from the introduction. 

\begin{proof}[Proof of Theorem \ref{thm: finiteness}] By Theorem \ref{theorem: atoroidal}, the number of tight contact structures up to isotopy on $M$ is finite since $M$ is hyperbolic. 
Hence (for fixed $\gamma$) the possibilities for the pullback $\xi_{\pm}^\gamma$ are finite. These contact structures however encode whether $\gamma$ is periodic. Thus the number of maps

$$\{\text{(unoriented) free homotopy classes } [\gamma^{\pm 1}] \in \pi_1(M)\} \longrightarrow \{ \text{P, NP} \} $$
depending on whether the class is periodic or not -- coded by $P$ or $NP$ -- is finite. More precisely, for some choice of longitudes
$$[\gamma^{\pm 1}] \longmapsto \begin{cases}
    P \ \ \ \  , \  Sl_{\xi_+}(\gamma,l) = Sl_{\xi_-}(\gamma,l) \text{ constant at infinity } \\
    NP \ , \ \text{ otherwise.}
\end{cases}$$
In particular, there are finitely many possibilities for the periodic set
$$\mathcal{P}er(X) = \{ [\gamma] \ | \ \gamma \text{ or } \gamma^{-1} \text{ periodic } \}$$
Since the flow is determined up to isotopy by this set in view of \cite{BFM} (cf.\ Theorem \ref{theorem: BFM}), we deduce that there are finitely many Anosov Flows up to isotopy.
\end{proof}
\begin{remark} Thanks to our covering arguments, cf.\ Remark \ref{rem:Finite_Index}, 
the proof of Theorem \ref{thm: finiteness} only uses the finiteness of orientable tight contact structures. However, the proof of Theorem \ref{theorem: atoroidal} from \cite{CGH}, relying on a normalization process with respect to a triangulation, extends straightforward to the non orientable case, in particular since plane fields are always orientable in restriction to every simplex.
Thus the bound on isotopy orbit equivalence classes of Anosov flows in terms of the number of isotopy classes of universally tight structures remains valid even when the stable and unstable foliations are not orientable, that is when the associated contact pair is not orientable.
    \end{remark}

\subsection{Finiteness of Periodic Orbits on Hyperbolic JSJ-pieces}

Although we cannot obtain Theorem \ref{thm: homotopy} for general flows on toroidal manifolds, the arguments are sufficient to show that the contact structures determines which non-peripheral elements in the fundamental group of a hyperbolic pieces are freely homotopic to periodic orbits.

\begin{proof}[Proof of Theorem \ref{thm: finiteness_JSJ}] Now since the bi-contact structures associated to the Anosov flow are universally tight and have vanishing Giroux torsion (cf. \cite{BBM}) there are finitely many up to contactomorphism by \cite[Th\'{e}or\`{e}me 6]{CGH}, see also Theorem \ref{theorem: toroidal}. More precisely, there are finitely many tight contact structures up to isotopy and {\em Lutz} modifications, which are in turn given by (compositions) of Dehn twists along embedded incompressible tori (not necessarly pairwise disjoint). 

We may therefore assume (up to finite ambiguity) that for the given flows $X,X'$ we have $\xi_+ = \xi'_+$. We may also assume that $\xi_-'$ is obtained from $\xi_-$ via a diffeomorphism $h$ that is a composition of Dehn twists on some finite collection $T_1, \cdots ,T_N$ of embedded tori and isotopy.

 We lift these tori to the cyclic cover $M^\gamma$. If $\gamma \in \pi_1(M_{hyp})$ is in a hyperbolic piece, then $h_*\gamma =\gamma$ and $h$ lifts to a 
 diffeomorphism $h^\gamma$ of $M^\gamma$ that takes $\xi_-^\gamma$ to $(\xi'_-)^\gamma$. Furthermore, as the support of $h$ is contained in a neighbourhood of the boundary, we can assume that $h^\gamma$ fixes some core curve of $M^\gamma$. In particular $h^\gamma$ is (smoothly) isotopic to the identity so that $\xi_-^\gamma$ and $(\xi'_-)^\gamma$ have the same slope at infinity. We conclude that whenever $\xi_+=\xi_+'$, then an element $\gamma \in \pi_1(M_{hyp})$ is periodic for $X$ if and only if it is for $X'$, from which the claimed finiteness follows.
\end{proof}

\subsection{Pseudo-Anosov Flows}
 In order to prove the finiteness of pseudo-Anosov flows on a closed hyperbolic $3$-manifold up to isotopy orbit equivalence, one can restrict to the case of pseudo-Anosov flows that share the same singular periodic orbits since Li \cite{Li} showed that the possibilities for singular orbits are finite (up to isotopy). In particular, this means that the Finiteness Conjecture on closed hyperbolic $3$-manifolds reduces to the case of certain link complements.

\begin{proof}[Proof of Theorem \ref{thm:finiteness_pA_case}]
In analogy to the Franks-Williams example given by performing a Derived from Anosov modification on certain periodic orbits and doubling, we do the same at all {\em singular orbits} to obtain an Anosov flow on a certain double. To be precise we perform an attracting Derived from Pseudo-Anosov blow up (cf.\ Lemma 11.1 \cite{BBY}) at all singular orbits which gives a hyperbolic plug in the sense of Beguin-Bonatti-Yu on the manifold $M_{reg}$ obtained by deleting small open nieghbourhoods of the singular orbits that is transverse an outward pointing at the boundary. We glue two copies of this after reversing the orientation and direction of the flow on one. This then gives an Anosov flow on the double $DM_{reg}$ which has two hyperbolic JSJ-pieces. By Theorem \ref{thm: finiteness_JSJ} the possibilities for homotopy classes of non-singular periodic orbits on $M_{reg}$ is finite and this determines the (non-singular) periodic orbit data on $M$. Since the possibilities for the singular orbits are also finite by Li's result we can then apply the main result of \cite{BFM} to deduce finiteness.
\end{proof}

\section{Outlook: Further Questions}\label{section: questions}

    We conclude with some questions that arise naturally in this setting.

\begin{question}[Slopes and Quasimorphisms] We have seen in Remark \ref{remark: relative} that the relative slope map verifies:
$$Sl_{\xi_1} (\gamma^n,\xi_2)=\vert n\vert Sl_{\xi_1} (\gamma,\xi_2).$$
What are the further properties? Does it define a pseudo-distance on $\pi_1(M)$? Are there circumstances where we can get rid of the absolute values and obtain a quasi-morphism on $\pi_1(M)$? It seems to be the case in the $\R$-covered situation. Is this map related to representations of $\pi_1(M)$ in $Homeo (\R)$?
\end{question}

The slope map only uses a small amount of the information given by the pair $(M^\gamma, \xi^\gamma)$. There is also the whole Giroux normal form \cite{Gi2} that could be used, or the contact homologies (cylindrical or embedded), obtained as direct limits of contact homologies of exhaustions by contact sutured tori.
\begin{question}[Detecting Lozenges via bi-contact structure] In the periodic case, does the number of dividing curves (for an exhaustion of $M^\gamma$) at infinity recover the number of lozenges?
\end{question}
Our techniques should yield that the number of dividing curves is at least the number of lozenges. This would then give a characterisation of the skew property purely in terms of the contact structures.
\begin{question}[Completions of $M^\gamma$]  Can one construct a boundary at infinity, endowed with some characteristic foliation extra data as a completion or compactification of $(M^\gamma, \xi^\gamma_\pm)$?
\end{question}
Finally one would like to extend our results to prove the Finiteness Conjecture in full generality.
    \begin{problem}[General pseudo-Anosov Flows]
        Prove the Finiteness Conjecture for pseudo-Anosov flows using contact geometry for closed hyperbolic manifolds. Extend to the toroidal case.
    \end{problem}

\end{document}